\documentclass[11pt, reqno]{amsart}

\usepackage[english]{babel}
\usepackage[left=3.5cm, right=3.5cm, top=3.2cm, bottom=3.2cm]{geometry}
\usepackage{url, amsfonts, amsmath, amsthm, amssymb, mathtools, mathrsfs, enumerate}
\usepackage{color, xcolor, verbatim, tensor, tikz, tikz-cd}
\usetikzlibrary{angles,quotes}
\usepackage[colorlinks=true,citecolor=blue]{hyperref}
\hypersetup{linkcolor=magenta, linktoc=page}
\usetikzlibrary{matrix,calc}
\definecolor{wine-stain}{rgb}{0.5,0,0}

\newcommand{\red}[1]{\textcolor{red}{#1}}

\newtheorem{thm}{Theorem}[section]
\newtheorem{prop}[thm]{Proposition}
\newtheorem{lem}[thm]{Lemma}
\newtheorem{cor}[thm]{Corollary}

\theoremstyle{definition}

\newtheorem*{ques*}{Question}
\newtheorem*{thm*}{Theorem}
\newtheorem*{rem*}{Remark}
\newtheorem*{rems*}{Remarks}
\newtheorem*{exs*}{Examples}
\newtheorem*{mthm*}{Main Theorem}

\numberwithin{equation}{section}

\newcommand{\la}{\lambda}

\newcommand{\RR}{\mathbb{R}}

\newcommand{\sL}{\mathcal{L}}

\newcommand{\dvol}{\mathrm{dvol}}

\newcommand{\Rc}{\mathrm{Ric}}

\newcommand{\Div}{\mathrm{div}}

\makeatletter
\renewcommand*{\eqref}[1]{%
	\hyperref[{#1}]{\textup{\tagform@{\ref*{#1}}}}%
}
\makeatother

\title[Generalized $m$-quasi-Einstein manifolds of Yamabe-type]{On the rigidity of generalized $m$-quasi-Einstein manifolds of Yamabe-type}
\author[Ramesh Mete]{Ramesh Mete}
\address{Department of Mathematics, Indian Institute of Technology Bombay, Powai, Mumbai - 400076, India.}
\email{ramesh2025m@gmail.com, rameshm@math.iitb.ac.in}
\date{\today}

\begin{document}
	
\maketitle
	
\begin{abstract}
Motivated by the concept of almost Yamabe solitons, a special class of generalized $m$-quasi-Einstein manifolds is investigated in this paper. We refer to these Riemannian manifolds as {\em generalized $m$-quasi-Einstein manifolds of Yamabe-type}. We study the rigidity properties for the potential (or defining) vector field associated to these manifolds in both the compact and non-compact settings. We show that under certain natural assumptions the potential vector field either vanishes identically or becomes a non-trivial Killing vector field.
\end{abstract}

\section{Introduction}
\label{sec:introduction}
	
Suppose $(M, g)$ is a Riemannian manifold of dimension $n$. We always assume that the manifold $M$ is connected and without boundary. Let $X \in \mathfrak{X}(M)$ be a smooth vector field on $M$ and let us fix a non-zero constant $m \in \RR \setminus \{0\}$. We denote by $X^{\flat}$ the dual $1$-form associated to $X$ with respect to the Riemannian metric $g$, i.e. $X^{\flat}(Y) := g(Y, X)$ for all smooth vector fields $Y \in \mathfrak{X}(M)$.

\vspace*{1mm}
Given a real-valued smooth function $\la: M \to \RR$, the quadruple $(M, g, X, \la)$ is called a {\em generalized $m$-quasi-Einstein ($m$-QE) manifold} if it solves the following equation
\begin{equation}\label{eq:gen-m-quasi-Ein}
\Rc + \frac{1}{2}\sL_{X} g - \frac{1}{m} X^{\flat}\otimes X^{\flat} = \la g,
\end{equation}
where $\sL_{X}$ denotes the Lie derivative along the smooth vector field $X$ and $\Rc$ stands for the Ricci curvature of the metric $g$. As in \cite{BarrosRibeiro2012}, the tensor quantity on the left-hand side in \eqref{eq:gen-m-quasi-Ein} is sometimes denoted by $\Rc_{X}^{m}$. Moreover, when $X := \nabla f$ for some smooth function $f \in C^\infty(M, \RR)$ this $(0, 2)$-tensor is the so-called {\em $m$-Bakry-Emery Ricci tensor}, denoted by $\Rc_{f}^{m}$, and it is given as follows:
\begin{align*}
\Rc_{f}^{m} = \Rc + \nabla^{2} f - \frac{1}{m} df \otimes df.
\end{align*}
Here, $\nabla$ denotes the {\em Levi-Civita connection} of the Riemannian manifold $(M, g)$. Note that for $X := \nabla f$ the second term involving the Lie derivative in \eqref{eq:gen-m-quasi-Ein} becomes the {\em Hessian} of the potential function $f$, denoted by $\nabla^{2} f$ or $\text{Hess}(f)$, and we also have $X^{\flat} = df$ for the third term on the left-hand side in \eqref{eq:gen-m-quasi-Ein}. The smooth vector field $X$ is said to be the {\em potential (or defining) vector field} for the generalized $m$-QE manifold $(M, g, X, \la)$ and the metric $g$ satisfying \eqref{eq:gen-m-quasi-Ein} will be called a {\em generalized $m$-QE metric}.

\vspace*{1mm}
We now mention several special cases of the above equation. Firstly, when $X := \nabla f$ for some $f\in C^{\infty}(M, \RR)$ the quadruple $(M, g, \nabla f, \la)$ is called a {\em generalized gradient $m$-QE manifold} and the function $f$ is called a {\em potential function}. For simplicity, a generalized gradient $m$-QE manifold is also denoted by $(M, g, f, \la)$. Secondly, if the function $\la$ in the equation \eqref{eq:gen-m-quasi-Ein} is {\em constant} (i.e. $\la \in \RR$), then the quadruple $(M, g, X, \la)$ is called a {\em $m$-QE manifold}, and further it is said to be a {\em gradient $m$-QE manifold} when $X = \nabla f$. Next, we say that a generalized $m$-QE manifold is {\em trivial} if the potential vector field $X$ vanishes identically (i.e. $X \equiv 0$). Otherwise, it will be called {\em non-trivial}. In the case $X \equiv 0$, one can easily see that the equation \eqref{eq:gen-m-quasi-Ein} is reduced to
\begin{equation*}
\Rc = \la g
\end{equation*}
and hence $(M, g)$ is an {\em Einstein manifold} when the dimension $n \geq 3$ by the well-known Schur's lemma for the Ricci tensor. Lastly, when $m  \to \pm\infty$ the equation \eqref{eq:gen-m-quasi-Ein} becomes
\begin{align*}
\Rc + \frac{1}{2} \sL_{X} g = \la g.
\end{align*}
This is the so-called {\em Ricci almost soliton equation}, which was first introduced and studied in \cite{Pigola-et-al-2011}. When $\la$ is a fixed {\em constant} (i.e. $\la \in \RR$) these are called {\em Ricci solitons}, which are investigated extensively in the past few decades. Note that Ricci solitons are self-similar solutions of the celebrated Ricci flow introduced by Hamilton in his seminal paper \cite{Hamilton1982}.

\vspace*{1mm}
In \cite{Catino2012}, Catino defined that a complete Riemannian manifold $(M^n, g)$, $n \geq 3$, is a {\em generalized quasi-Einstein manifold}, if there exist three smooth functions $f, \mu, \la$ on $M$ such that
\begin{align*}
\Rc + \nabla^{2} f - \mu df \otimes df = \la g.
\end{align*}
Later in \cite{BarrosRibeiro2014}, Barros-Ribeiro Jr introduced the notion of {\em generalized $m$-quasi-Einstein ($m$-QE) manifolds} in the sense of equation \eqref{eq:gen-m-quasi-Ein} (assuming that $0 < m \leq \infty$ is an integer) and they also gave some non-trivial examples of generalized $m$-QE metrics on the standard unit $n$-sphere $\mathbb{S}^n$. It was proved in \cite[Theorem 1]{Barros-Gomes2013} that if $(M^n, g, \nabla f, \la)$ is a nontrivial compact generalized $m$-QE manifold with $n \geq 3$ (and $0 < m \leq \infty$ is an integer) satisfying $\sL_{\nabla u}R \geq 0$, where $u := e^{- f/m}$, then the scalar curvature $R$ is constant and $M^n$ is isometric to a standard sphere $\mathbb{S}^n(r)$. There are also several integral formulae for compact generalized $m$-QE manifolds, for instance, see \cite[Theorem 4]{BarrosRibeiro2014} and \cite[Proposition 1, 2]{Barros-Gomes2017}. The reader is referred to the following articles \cite{Ghosh2020, Guler-De2022, PoddShaBal2023-gen-m-QE, Ghosh2026, CollingDunajski2026, Cochran2025arXiv, CarvLimCosta-Filho2026} for some recent results about generalized $m$-QE manifolds.

\vspace*{1mm}
One can view generalized $m$-QE manifolds as certain generalization of Ricci solitons (and {\em almost Yamabe solitons} which are defined below). These manifolds are connected with general relativity. More precisely, they provide mathematical frameworks which are used to model extremal black holes by generalizing the standard near-horizon geometry. For this reason, some authors prefer to call them as {\em generalized extremal horizon manifolds/equations} (for instance, see \cite{BGKW2022, Wylie2023, CollingDunajski2026}).

\vspace*{1mm}
A smooth vector field $Y \in \mathfrak{X}(M)$ is said to be {\em conformal} if 
\begin{equation}\label{eq:conformal-vector-field}
\frac{1}{2} \sL_{Y} g = \psi g
\end{equation}
for some real-valued smooth function $\psi$ on $M$; and it is said to be {\em Killing} if the conformal factor $\psi$ is identically zero on $M$. When $\psi$ is any constant (not just identically zero) the conformal vector field $Y$ is called {\em homothetic}. We say that $Y \in \mathfrak{X}(M)$ is a {\em geodesic vector field} if $\nabla_{Y}Y = 0$ (i.e. integral curves of $Y$ are geodesics). Recall that $Y$ is called {\em closed} if $\nabla_{Z}Y = \psi Z$ for all $Z \in \mathfrak{X}(M)$ and some function $\psi \in C^{\infty}(M, \RR)$; and in particular, it is called {\em parallel} if the function $\psi$ is identically zero. Note that closed vector fields are conformal; and all parallel vector fields are geodesic vector fields. In \cite[Theorem 1]{Barros-Gomes2017}, Barros-Gomes proved that any {\em compact} $m$-QE manifold $(M^n, g, X, \la)$ of dimension $n \geq 3$ (with $0 < m \leq \infty$ is an integer) must be trivial (i.e. $X \equiv 0$) if $(M^n, g)$ is itself an Einstein manifold. Moreover, they showed that the potential vector field $X$ for any compact $m$-QE manifold must be Killing if it is a conformal vector field (cf. \cite[Corollary 1]{Barros-Gomes2017}).

\vspace*{1mm}
In this article, we restrict our attention to certain class of generalized $m$-QE manifolds $(M^n, g, X, \la)$ with $\la := R - \rho$, where $R$ denotes the scalar curvature of the metric $g$ and $\rho: M \to \RR$ is a smooth function. In other words, we consider the following equation
\begin{equation}\label{eq:gen-m-QE-Yam-type}
\Rc + \frac{1}{2}\sL_{X} g - \frac{1}{m} X^{\flat}\otimes X^{\flat} = (R - \rho) g.
\end{equation}
In this case, we say that $g$ is a {\em generalized $m$-QE metric of Yamabe-type}. Observe that if the Ricci tensor satisfies 
$$\Rc = \frac{1}{m} X^{\flat}\otimes X^{\flat},$$ 
then the equation \eqref{eq:gen-m-QE-Yam-type} becomes 
\begin{equation}\label{eq:almost-Yam-soliton}
\frac{1}{2}\sL_{X} g = (R - \rho) g.
\end{equation}
Any metric $g$ satisfying the equation \eqref{eq:almost-Yam-soliton} is called an {\em almost Yamabe soliton}, which was first introduced and studied by Barbosa-Ribeiro Jr \cite{Barbosa-Ribeiro2013}. When $\rho$ is a constant (i.e. $\rho \in \RR$), it is said to be a {\em Yamabe soliton}. It was proved in \cite[Proposition 2.1]{Barbosa-Ribeiro2013} that almost Yamabe solitons are conformal solutions of the {\em non-normalized Yamabe flow}, introduced by Hamilton \cite{Hamilton1986}. Similar to Ricci solitons, the classification and rigidity results for Yamabe solitons have been studied extensively; for instance, see \cite{di Cerbo-Disconzi2008, Hsu2012, Cao-Sun-Zhang2012, Daskalo-Sesum2013, Maeta2019, Maeta2020}. For some recent results about almost Yamabe solitons one can see \cite{Barbosa-Ribeiro2013, PoddShaBal2024-almost-Yam, HwangYun2025arXiv, Mete2026-almost-Yam} and references therein.

\vspace*{1mm}
We now state our main results in the compact case. Throughout the paper, given a Riemannian metric $g$ we use the notation $\dvol_{g}$ to denote the {\em  Riemannian volume element} corresponding to $g$ and the notation $\langle \bullet, \bullet \rangle$ to denote the quantity $g(\bullet, \bullet)$. We say that a manifold is {\em closed} if it is compact and without boundary.
	
\begin{thm}\label{thm:parallel-conclusion-for-geod-vec-field-on-gen-m-QE-Yam-type}
Let $(M, g, X, \rho)$ be a closed generalized $m$-quasi-Einstein manifold of Yamabe-type in the sense of equation \eqref{eq:gen-m-QE-Yam-type}. Assume that $\rho \in C^\infty(M, \RR)$ satisfies
\begin{equation}\label{eq:integration-of-X-rho-nonpositive-assumption}
\int_{M} \langle X, \nabla \rho \rangle \dvol_{g} \leq 0.
\end{equation}
Then the potential vector field $X$ is parallel if any of the following conditions holds:
\begin{enumerate}[~~~(i)]
\item The norm $\lvert X \rvert$ is constant and
\begin{align}\label{eq:integration-of-Ric-X-X-is-nonpositive}
\int_{M} \Rc(X, X) \dvol_{g} \leq 0;
\end{align}
	
\item $X$ is a geodesic vector field and the inequality \eqref{eq:integration-of-Ric-X-X-is-nonpositive} holds.
\end{enumerate}
\end{thm}
	
The same conclusion as in Theorem \ref{thm:parallel-conclusion-for-geod-vec-field-on-gen-m-QE-Yam-type} holds for conformal potential vector fields but we do not need the assumption \eqref{eq:integration-of-X-rho-nonpositive-assumption}.
	
\begin{prop}\label{prop:parallel-conclusion-for-conf-vec-field-on-gen-m-QE-Yam-type}
Let $(M, g, X, \rho)$ be a closed generalized $m$-quasi-Einstein manifold of Yamabe-type in the sense of equation \eqref{eq:gen-m-QE-Yam-type} for some $\rho \in C^\infty(M, \RR)$. If the potential vector field $X$ is conformal and the inequality \eqref{eq:integration-of-Ric-X-X-is-nonpositive} holds, then $X$ is parallel.
\end{prop}
	
One can relate Theorem \ref{thm:parallel-conclusion-for-geod-vec-field-on-gen-m-QE-Yam-type} and Proposition \ref{prop:parallel-conclusion-for-conf-vec-field-on-gen-m-QE-Yam-type} with Corollary $5$ in \cite{CarvLimCosta-Filho2026}. For a closed generalized $m$-QE manifold of Yamabe-type, this corollary states that under the following assumption
\begin{align}\label{eq:integration-of-X-lambda-nonpositive-assumption-for-m-QE-Yam-type}
\int_{M} \langle X, \nabla R \rangle \dvol_{g} \leq \int_{M} \langle X, \nabla \rho \rangle \dvol_{g} ~~~ \left(\iff ~ \int_{M} \langle X, \nabla \la \rangle \dvol_{g} \leq 0\right)
\end{align}
the potential vector field $X$ is parallel if one of the following conditions holds:
\begin{enumerate}[~~~(1)]
\item $\int_{M} \Rc(X, X) \dvol_{g} \leq \frac{2}{m} \int_{M} \lvert X \rvert^{2} \Div(X) \dvol_{g}$;
		
\vspace*{1mm}
\item $X$ is conformal and $\int_{M} \Rc(X, X) \dvol_{g} \leq 0$;
		
\vspace*{1mm}
\item $\lvert X \rvert$ is constant and $\int_{M} \Rc(X, X) \dvol_{g} \leq 0$.
\end{enumerate}
As in Proposition \ref{prop:parallel-conclusion-for-conf-vec-field-on-gen-m-QE-Yam-type}, we remark that the assumption \eqref{eq:integration-of-X-lambda-nonpositive-assumption-for-m-QE-Yam-type} is not required for the desired conclusion when $X$ is conformal.

\vspace*{1mm}
Carvalho, Lima and Costa-Filho \cite{CarvLimCosta-Filho2026} proved a rigidity result for closed {\em generalized} $m$-QE manifolds by extending the corresponding result for closed $m$-QE manifolds due to Sharma \cite{Sharma2025} . In the setting of equation \eqref{eq:gen-m-QE-Yam-type}, it can be re-stated as follows:
	
\begin{thm}[{\cite[Theorem A]{CarvLimCosta-Filho2026}}]
Let $(M^n, g, X, \rho)$ be a closed generalized $m$-quasi-Einstein manifold of Yamabe-type in the sense of equation \eqref{eq:gen-m-QE-Yam-type}. Assume that the scalar curvature $R$ of the metric $g$ satisfies
\begin{equation*}
\int_{M} \langle X, \nabla R \rangle \dvol_{g} \leq 0
\end{equation*}
and the following condition holds
\begin{align*}
\int_{M} \langle X, \nabla R \rangle \dvol_{g} \geq \int_{M} \langle X, \nabla \rho \rangle \dvol_{g} ~~~ \left(\iff ~ \int_{M} \langle X, \nabla \la \rangle \dvol_{g} \geq 0\right).
\end{align*}
Then $X$ is a Killing vector field.
\end{thm}
	
We have an analogous result for any generalized $m$-QE manifold of Yamabe-type.
	
\begin{thm}\label{thm:Killing-of-gen-m-QE-Yam-type-under-certain-condition}
Let $(M^n, g, X, \rho)$ be an $n$-dimensional closed generalized $m$-quasi-Einstein manifold of Yamabe-type in the sense of equation \eqref{eq:gen-m-QE-Yam-type} for some $\rho \in C^\infty(M, \RR)$. Then the potential vector field $X$ is Killing if any of the following criterions holds:
\begin{enumerate}[~~~(i)]
\item The scalar curvature $R$ of the metric $g$ satisfies
\begin{equation}
\label{eq:integral-of-X-rho-less-than-some-constant-times-integral-of-X-R}
\int_{M} \langle X, \nabla \rho \rangle \dvol_{g} \leq \frac{n}{n+2} \int_{M} \langle X, \nabla R \rangle \dvol_{g}.
\end{equation}
	
\item The function $\rho$ satisfies the condition \eqref{eq:integration-of-X-rho-nonpositive-assumption} and the scalar curvature $R$ satisfies 
\begin{equation}\label{eq:integral-of-X-R-nonnegative}
\int_{M} \langle X, \nabla R \rangle \dvol_{g} \geq 0.
\end{equation}

\item The potential vector field $X$ is conformal and the condition \eqref{eq:integration-of-X-rho-nonpositive-assumption} holds.
\end{enumerate}
\end{thm}

In \cite{CarvLimCosta-Filho2026} (see also \cite{Ghosh2026}), it was proved that if $\Div(X) = 0$ on a closed generalized $m$-QE manifold (i.e. equation \eqref{eq:gen-m-quasi-Ein} holds), then $X$ must be Killing. The same fact for closed generalized $m$-QE manifold of Yamabe-type can be obtained as a corollary of Theorem \ref{thm:Killing-of-gen-m-QE-Yam-type-under-certain-condition} since the assumption $\Div(X) = 0$ implies that 
\begin{align*}
\int_{M} \langle X, \nabla R \rangle \dvol_{g} = \int_{M} \langle X, \nabla \rho \rangle \dvol_{g} = 0.
\end{align*}
Moreover, under the condition $\Div(X) = 0$ the scalar curvature $R$ of the metric $g$ is constant if and only if the function $\la$ is constant (see Theorem C in \cite{CarvLimCosta-Filho2026}).

\vspace*{1mm}
It was proved in \cite[Corollary E]{CarvLimCosta-Filho2026} that if $m < 0$ and the function $\la \leq 0$, then any divergence-free potential vector field $X$ on a closed generalized $m$-QE manifold $(M, g, X, \la)$ must be identically zero. In the setting of equation \eqref{eq:gen-m-QE-Yam-type}, this fact can be generalized to the following result.
	
\begin{prop}\label{prop:div-free-plus-upper-bound-on-rho-implies-triviality}
Let $(M^n, g, X, \rho)$ be a closed generalized $m$-quasi-Einstein manifold of Yamabe-type (i.e. equation \eqref{eq:gen-m-QE-Yam-type} holds) of dimension $n \geq 2$. Assume that $\Div(X) = 0$ and the smooth function $\rho$ satisfies the following inequality
\begin{equation}\label{eq:upper-bound-assumption-for-rho}
\rho <  \frac{2 - n}{m} \lvert X \rvert^{2}.
\end{equation}
Then the potential vector field $X$ vanishes identically on $M$. In particular, if $m < 0$ and $\rho \leq 0$ for $n\geq 3$ (or $\rho < 0$ for $n =2$), then $X \equiv 0$ provided $X$ is divergence-free.
\end{prop}
	
As a consequence of above proposition, we obtain the following corollary.
	
\begin{cor}\label{cor:div-free-plus-scal-curv-nonpositive-implies-triviality}
Let $(M^n, g, X, \rho)$ be a closed generalized $m$-quasi-Einstein manifold of Yamabe-type (i.e. equation \eqref{eq:gen-m-QE-Yam-type} holds) of dimension $n \geq 2$. Assume that $\Div(X) = 0$. Also, suppose that one of the following conditions holds:
\begin{enumerate}[~~~(i)]
\item The constant $m < 0$ and the scalar curvature $R$ is non-positive;
			
\vspace*{1mm}
\item The constant $m < 0$ and $R \leq \rho$ (i.e. $\la \leq 0$).
\end{enumerate}
Then the potential vector field $X$ vanishes identically on $M$. 
\end{cor}

\vspace*{1mm}
For each $p \geq 1$, we use the notation $L^{p}(M, g)$ (or simply $L^{p}(M)$) to denote the space of functions $f: M \rightarrow \RR$ on the Riemannian manifold $(M, g)$ such that $$\int_{M} \vert f \vert^{p} \dvol_{g} < + \infty.$$ We have the following results when the manifold is non-compact.
	
\begin{thm}\label{thm:triviality-of-gen-m-QE-mfd-complete-noncompact-case}
Let $(M, g, X, \rho)$ be a complete, non-compact generalized $m$-quasi-Einstein manifold of Yamabe-type in the sense of equation \eqref{eq:gen-m-QE-Yam-type} of dimension $n \geq 3$. Assume that the potential vector field $X$ satisfies $\lvert X \rvert \in L^1(M)$ and the following condition holds:
\begin{equation}\label{eq:condition-implies-X-is-zero-complete-noncompact-case}
m \Big((n-1) R - n \rho \Big) \geq 0.
\end{equation}
Then $X$ is identically zero on $M$, and hence $(M, g)$ is an Einstein manifold.
\end{thm}

\begin{thm}\label{thm:parallel-vec-field-conclusion-gen-mQE-mfd-Yam-type-non-compact-case}
Let $(M, g, X, \rho)$ be a complete, non-compact generalized $m$-quasi-Einstein manifold of Yamabe type in the sense of equation \eqref{eq:gen-m-QE-Yam-type} of dimension $n \geq 3$. Assume that the smooth function $\rho$ satisfies the following inequality
\begin{equation}\label{eq:condition-to-guarantee-div-of-V-is-nonnegative}
(n-2) \langle X, \nabla \rho \rangle \leq (1 - n)\Rc(X, X) - \frac{n}{m} X(\vert X \vert^{2}),
\end{equation}
and that the potential vector field $X$ satisfies any of the following conditions:
\begin{enumerate}[~~~(i)]
\item $\vert X \vert \in L^3(M)$, $\Div(X) \in L^{3/2}(M)$ and $\vert \nabla\vert X \vert \vert \in L^{3/2}(M)$;
		
\vspace*{1mm}
\item $\vert X \vert \in L^{3}(M) \cap W^{1,2}(M)$ and $\Div(X) \in L^{2}(M)$.
\end{enumerate}
Then $X$ is a parallel vector field.
\end{thm}

Here, the Sobolev (or Hilbert) space $W^{1,2}(M)$ is the completion of the vector space of smooth real-valued functions $f$ such that $f\in L^2(M, g)$ and $\vert \nabla f \vert_{g} \in L^2(M, g)$ with respect to the norm
$$ \Vert f \Vert_{W^{1,2}(M)} := \left(\Vert f \Vert_{L^{2}(M)}^{2} + \Vert \nabla f \Vert_{L^{2}(M)}^{2} \right)^{1/2}. $$

\vspace*{1mm}
It is worth noting that the study of generalized $m$-QE manifolds is significantly more challenging in non-compact settings than in the compact setting. It will be interesting if one finds analogous criterions as in \cite{HwangYun2025arXiv, Mete2026-almost-Yam} (for almost Yamabe solitons) to prove rigidity results on complete, non-compact generalized $m$-QE manifolds of Yamabe-type.

\vspace*{1mm}
\textbf{Brief outline of the paper:} In Section \ref{section:preliminaries}, we first present a key lemma about conformal vector fields on any closed Riemannian manifold and another key lemma about geodesic vector fields on closed generalized $m$-quasi-Einstein ($m$-QE) manifolds. After that in Section \ref{section:preliminaries}, we recall several basic properties and identities for generalized $m$-QE manifolds and we also recall Yano's integral formula. The Section \ref{subsec:proof of main results_compact case} contains the proofs of our main results from Section \ref{sec:introduction} in the compact setting; while Section \ref{subsec:proof of main results_non-compact case} is devoted to proving our results in the non-compact case.

\vspace*{2mm}
\section{Preliminaries}
\label{section:preliminaries}
	
We begin with the following observation for any conformal vector field on a closed Riemannian manifold (which need not to be a $m$-quasi-Einstein manifold).

\begin{lem}\label{lem:integration-of-mod-X-square-Div-X-is-zero-for-conformal-vec-field}
Let $(M^n, g)$ be a closed Riemannian manifold of dimension $n$ and $Y \in \mathfrak{X}(M)$ a smooth vector field. If $\lvert Y \rvert$ is constant or $Y$ is conformal, then we have
\begin{align*}
\int_M \lvert Y \rvert^{2} \Div(Y) \dvol_{g} = 0.
\end{align*}
\end{lem}

\begin{proof}
When the $\lvert Y \rvert$ is constant the desired conclusion follows from the divergence theorem since $M$ is closed. Now, suppose $Y$ is conformal. Then \eqref{eq:conformal-vector-field} holds for some $\psi \in C^{\infty}(M, \RR)$, and in particular, taking trace we obtain
\begin{equation}\label{eq:trace-of-conform-vec-field}
n \psi = \Div(Y).
\end{equation}
On the other hand, evaluating \eqref{eq:conformal-vector-field} at the pair $(Y, Y)$ it yields
\begin{equation*}
\langle \nabla_{Y}Y, Y \rangle = \frac{1}{2}(\sL_{Y} g)(Y, Y) = \psi \lvert Y \rvert^{2}.
\end{equation*}
But, using \eqref{eq:trace-of-conform-vec-field} and the metric compatibility property of $\nabla$ we get
$$ \frac{1}{n} \lvert Y \rvert^{2} \Div(Y) = \langle \nabla_{Y}Y, Y \rangle = \frac{1}{2} Y (\lvert Y \rvert^{2}).$$
Now, using the following divergence formula $$\Div(\lvert Y \rvert^{2} Y) = \lvert Y \rvert^{2} \Div(Y) + \langle \nabla \lvert Y \rvert^{2}, Y \rangle = \lvert Y \rvert^{2} \Div(Y) + Y (\lvert Y \rvert^{2})$$ and simplifying we obtain $$ (\frac{2}{n} + 1) \lvert Y \rvert^{2} \Div(Y) = \Div(\lvert Y \rvert^{2} Y).$$
Hence, applying the divergence theorem we get the desired result.
\end{proof}
	
Next, we see that same conclusion holds for any geodesic vector field on a closed generalized $m$-QE manifold.
	
\begin{lem}\label{lem:integration-of-mod-X-square-Div-X-is-zero-for-geod-vec-field}
Let $(M, g, X, \la)$ be a closed generalized $m$-QE manifold (i.e. the equation \eqref{eq:gen-m-quasi-Ein} holds) with a geodesic vector field $X$. Then 
\begin{align*}
\int_M \lvert X \rvert^{2} \Div(X) \dvol_{g} = 0.
\end{align*}
\end{lem}

\begin{proof}
The proof of this result is contained in the proof of Theorem F in \cite{CarvLimCosta-Filho2026}. But for reader's convenience we give a rough sketch of the proof. Since $X$ is a geodesic vector field, we have $\nabla_{X} X = 0$. Using this condition and from the equation \eqref{eq:gen-m-quasi-Ein} one obtain
$$ \frac{1}{2} \nabla \lvert X \rvert^{2} + 2 \Rc(X) = 2 \left( \la + \frac{1}{m} \lvert X \rvert^{2}\right) X $$ (see equation $(9)$ in \cite{CarvLimCosta-Filho2026}). Here, $\Rc(X)$ is interpreted as $\Rc(X, Y)= \langle \Rc(X), Y \rangle$ for any $Y\in \mathfrak{X}(M)$. The above identity implies that
\begin{equation*}
\frac{1}{2} \Delta \lvert X \rvert^{2} + 2 \Div(\Rc(X)) = 2 \left( \la + \frac{\lvert X \rvert^{2}}{m} \right) \Div(X) + 2 \langle \nabla \la, X \rangle + \frac{4}{m} \langle \nabla_{X} X, X \rangle.
\end{equation*}
The last term on the right-hand side is zero since $\nabla_{X} X = 0$. Now, using the formula $$\la\Div(X) + \langle \nabla \la, X \rangle = \Div(\la X)$$ and the divergence theorem on closed manifolds, we finally obtain the desired result.
\end{proof}

\vspace*{1mm}
We now recall some basic properties and identities hold on generalized $m$-QE manifold of Yamabe-type. Firstly, taking trace in the equation \eqref{eq:gen-m-QE-Yam-type} we obtain
\begin{equation*}
R + \Div(X) - \frac{1}{m} \vert X \rvert^{2} = n (R - \rho),
\end{equation*}
where $R$ denotes the scalar curvature of the metric $g$. We re-write this identity as follows
\begin{equation}\label{eq:identity-after-taking-trace-of-gen-mQE-Yam-type}
n \rho = (n - 1) R - \Div(X) + \frac{1}{m} \vert X \rvert^{2}.
\end{equation}
In particular, applying $\nabla$ on both sides and taking inner product with $X$ itself, we get
\begin{equation}\label{eq:identity-after-taking-trace-and-applying-nabla-and-inner-prod-with-X}
n \langle X, \nabla \rho \rangle = (n-1) \langle X, \nabla R \rangle - \langle \nabla \Div(X), X \rangle + \frac{1}{m} X (\lvert X\rvert^2).
\end{equation}
On the other hand, the equation \eqref{eq:gen-m-QE-Yam-type} on the pair $(X, X)$ gives us
\begin{equation*}
\Rc(X, X) + \langle \nabla_{X}X, X \rangle - \frac{1}{m} \lvert X \rvert^{4} = (R - \rho) \lvert X \rvert^{2}.
\end{equation*}
Since $2 \langle \nabla_{X}X, X \rangle = X(\lvert X \rvert^{2})$, it can be re-written as
\begin{equation}\label{eq:evaluating-gen-m-QE-Yam-type-on-the-pair-X-X}
\Rc(X, X) + \frac{1}{2} X (\lvert X \rvert^{2}) = \frac{1}{m} \lvert X \rvert^{4} + (R - \rho) \lvert X \rvert^{2}.
\end{equation}

\vspace*{1mm}
We also have another crucial identity on any generalized $m$-QE manifold due to Ghosh \cite{Ghosh2026}. Note that this identity was proved in \cite{BarrosRibeiro2014} for the gradient case, i.e. when $X = \nabla f$, and before that it was established for $m$-QE manifolds (i.e. when $\la$ is constant) in \cite{BarrosRibeiro2012}.
	
\begin{lem}
Let $(M, g, X, \la)$ be a generalized $m$-QE manifold of dimension $n$. Then
\begin{align}\label{eq:Bochner-type-identity-for-gen-m-QE}
\frac{1}{2} \Delta \lvert X \rvert^{2} = \lvert\nabla X \rvert^{2} - \Rc(X, X) + \frac{2}{m} \lvert X\rvert^{2}\Div(X) -(n-2)\langle X, \nabla \la \rangle.
\end{align}
In particular, for a generalized $m$-QE manifold of Yamabe-type we have
\begin{align}\label{eq:Bochner-type-identity-for-gen-m-QE-Yam-type}
\frac{1}{2} \Delta \lvert X \rvert^{2} = \lvert\nabla X \rvert^{2} - \Rc(X, X) + \frac{2}{m} \lvert X\rvert^{2}\Div(X) -(n-2)\langle X, \nabla (R - \rho) \rangle.
\end{align}
The last term on the right-hand side of both identities vanishes for any $m$-QE manifold. 
\end{lem}
	
The proof of identity \eqref{eq:Bochner-type-identity-for-gen-m-QE} for a generalized $m$-QE manifold is based on the following formula about the divergence of the Lie derivative of the metric along a vector field.
	
\begin{lem}[{\cite[Lemma 2.1]{PetWylie2009}}]
Suppose $Y$ is a smooth vector field on a Riemannian manifold $(M, g)$. Then
\begin{equation*}
\Div(\sL_{Y} g)(X) = \frac{1}{2}\Delta \lvert Y \rvert^{2} - \lvert\nabla Y \rvert^{2} + \Rc(Y, Y) + \nabla_{Y}\Div(Y).
\end{equation*}
\end{lem}
	
We next recall the well-known Yano's integral formula \cite{Yano1970-bk}, which relates the $L^2$-norm of the Lie derivative of a Riemannian metric along a smooth vector field to the Ricci curvature of that vector field. More precisely, if $Y \in \mathfrak{X}(M)$ is a smooth vector field on a {\em closed} Riemannian manifold $(M, g)$, then Yano's integral formula is given by
\begin{equation}\label{eq:Yano-integral-formula}
\frac{1}{2} \int_{M} \lvert \sL_{Y} g \rvert^{2} \dvol_{g} = \int_{M} \Big(\lvert\nabla Y \rvert^{2} + \Div(Y)^{2} - \Rc(Y, Y) \Big) \dvol_{g}.
\end{equation}
This classical formula is widely used to study the rigidity and non-existence of special vector fields. For instance, applying this formula one can see that any Killing vector field on a closed Riemannian manifold with non-positive Ricci curvature is parallel. In fact, there is no non-trivial Killing vector field on a closed Riemannian manifold with negative Ricci curvature.

\vspace*{1mm}
We end this section by recalling the following result due to Caminha \cite{Caminha2011}.
	
\begin{prop}[{\cite[Proposition 2.1]{Caminha2011}}]
\label{prop:Caminha-complete-noncompact-case}
Let $X$ be a smooth vector field on the complete, non-compact, oriented Riemannian manifold $M^n$, such that $\Div(X)$ does not change sign on $M$. If $\lvert X \rvert \in L^1(M)$, then $\Div(X) = 0$ on $M$.
\end{prop}

\vspace*{2mm}
\section{Proof of the results}

\subsection{Proof of the results in the compact setting}
\label{subsec:proof of main results_compact case}
	
In this sub-section we prove our main results when the manifold is compact.
	
\begin{proof}[\textbf{Proof of Theorem \ref{thm:parallel-conclusion-for-geod-vec-field-on-gen-m-QE-Yam-type}}]
Since the manifold $M$ is closed, integrating \eqref{eq:Bochner-type-identity-for-gen-m-QE-Yam-type} we get
\begin{align*}
\int_M \lvert \nabla X \rvert^{2} \dvol_{g} &= \int_{M} \Rc(X, X) \dvol_{g} - \frac{2}{m} \int_M \lvert X \rvert^{2} \Div(X) \dvol_{g} \\ & \hspace*{1cm} + (n - 2) \int_{M} \langle X, \nabla (R - \rho) \rangle \dvol_{g}.
\end{align*}
If $X$ is geodesic (or if $\lvert X\rvert$ is constant), then by Lemma \ref{lem:integration-of-mod-X-square-Div-X-is-zero-for-geod-vec-field} (or by Lemma \ref{lem:integration-of-mod-X-square-Div-X-is-zero-for-conformal-vec-field}) it yields that the second term on the right-hand side of above identity is zero. So, we have
\begin{align}\label{eq:identity-1-in-proof-of-Thm}
\int_M \lvert \nabla X \rvert^{2} \dvol_{g} = \int_{M} \Rc(X, X) \dvol_{g} + (n - 2) \int_{M} \langle X, \nabla (R - \rho) \rangle \dvol_{g}.
\end{align}
Now, integrating \eqref{eq:identity-after-taking-trace-and-applying-nabla-and-inner-prod-with-X} we obtain
\begin{align*}
n \int_M \langle X, \nabla \rho \rangle \dvol_{g} &= (n - 1) \int_M \langle X, \nabla R \rangle \dvol_{g} - \int_{M} \langle \nabla \Div(X), X \rangle \dvol_{g} \\ &\hspace*{2.5cm} + \frac{1}{m} \int_M X(\lvert X \rvert^{2}) \dvol_{g}.
\end{align*}
Using the following formulae
\begin{equation}\label{eq:two-divergence-formulae}
\begin{split}
\Div(\Div(X)\cdot X) &= \Div(X)^2 + \langle \nabla \Div(X), X \rangle; \\
\Div(\lvert X \rvert^{2} X) &= \lvert X \rvert^{2} \Div(X) + X(\lvert X \rvert^{2}),
\end{split}
\end{equation}
and applying the divergence theorem and using Lemma \ref{lem:integration-of-mod-X-square-Div-X-is-zero-for-geod-vec-field} again, it yields that
\begin{align*}
n \int_M \langle X, \nabla \rho \rangle \dvol_{g} = (n - 1) \int_M \langle X, \nabla R \rangle \dvol_{g} + \int_{M} \Div(X)^{2} \dvol_{g}.
\end{align*}
We re-write this identity as follows
\begin{align*}
(n - 1) \int_M \langle X, \nabla (R - \rho) \rangle \dvol_{g} = \int_M \langle X, \nabla \rho \rangle \dvol_{g} - \int_{M} \Div(X)^{2} \dvol_{g}.
\end{align*}
Therefore, from the hypothesis \eqref{eq:integration-of-X-rho-nonpositive-assumption} we have
$$ \int_M \langle X, \nabla (R - \rho) \rangle \dvol_{g} \leq 0.$$
Then, from \eqref{eq:identity-1-in-proof-of-Thm} and the assumption \eqref{eq:integration-of-Ric-X-X-is-nonpositive} on the Ricci tensor it implies that
$$ \int_M \lvert \nabla X \rvert^{2} \dvol_{g} \leq \int_{M} \Rc(X, X) \dvol_{g} \leq 0. $$
Thus, we must have $\nabla X = 0$, i.e. $X$ is a parallel vector field. So, the proof is completed.
\end{proof}

Next, we prove our Proposition \ref{prop:parallel-conclusion-for-conf-vec-field-on-gen-m-QE-Yam-type}.
	
\begin{proof}[{\textbf{Proof of Proposition \ref{prop:parallel-conclusion-for-conf-vec-field-on-gen-m-QE-Yam-type}}}]
Since the manifold $M$ is compact and the potential vector field $X$ is conformal, applying Theorem II.9 of Bourguignon-Ezin \cite[pp. 727]{BourEzin1987} we have 
\begin{equation}\label{eq:integral-condition-obtained-by-BourEzin-theorem}
\int_{M} \langle X, \nabla R \rangle \dvol_{g} = 0.
\end{equation}
Now as in the proof of Theorem \ref{thm:parallel-conclusion-for-geod-vec-field-on-gen-m-QE-Yam-type}, since $X$ is conformal, we may apply Lemma \ref{lem:integration-of-mod-X-square-Div-X-is-zero-for-conformal-vec-field} to obtain the same identity \eqref{eq:identity-1-in-proof-of-Thm}, i.e.
\begin{equation*}
\int_M \lvert \nabla X \rvert^{2} \dvol_{g} = \int_{M} \Rc(X, X) \dvol_{g} + (n - 2) \int_{M} \langle X, \nabla (R - \rho) \rangle \dvol_{g},
\end{equation*}
and using \eqref{eq:integral-condition-obtained-by-BourEzin-theorem} it is now simplified to
\begin{equation*}
\int_M \lvert \nabla X \rvert^{2} \dvol_{g} = \int_{M} \Rc(X, X) \dvol_{g} - (n - 2) \int_{M} \langle X, \nabla \rho \rangle \dvol_{g}.
\end{equation*}
On the other hand, as in the proof of Theorem \ref{thm:parallel-conclusion-for-geod-vec-field-on-gen-m-QE-Yam-type}, we also have
\begin{align*}
n \int_M \langle X, \nabla \rho \rangle \dvol_{g} &= (n - 1) \int_M \langle X, \nabla R \rangle \dvol_{g} + \int_{M} \Div(X)^{2} \dvol_{g} \\ &= \int_{M} \Div(X)^{2} \dvol_{g},
\end{align*}
where we used \eqref{eq:integral-condition-obtained-by-BourEzin-theorem} again in the second line. It implies that
\begin{equation*}
\int_M \lvert \nabla X \rvert^{2} \dvol_{g} = \int_{M} \Rc(X, X) \dvol_{g} - \frac{n-2}{n} \int_{M} \Div(X)^2 \dvol_{g}.
\end{equation*}
Now, the assumption \eqref{eq:integration-of-Ric-X-X-is-nonpositive} on the Ricci tensor implies that $$ \int_M \lvert \nabla X \rvert^{2} \dvol_{g} \leq 0,$$ and hence $\nabla X = 0$. So, $X$ is a parallel vector field and the proof is completed.
\end{proof}

We now present the proof of Theorem \ref{thm:Killing-of-gen-m-QE-Yam-type-under-certain-condition}.
	
\begin{proof}[{\textbf{Proof of Theorem \ref{thm:Killing-of-gen-m-QE-Yam-type-under-certain-condition}}}]
The proof is basically an appropriate modification in the proof of Theorem A in \cite{CarvLimCosta-Filho2026}. Since the quadruple $(M, g, X, \rho)$ is a closed generalized $m$-QE manifold of Yamabe-type, integrating the identity \eqref{eq:Bochner-type-identity-for-gen-m-QE-Yam-type} we obtain
\begin{align*}
\int_{M} \Big(\lvert\nabla X \rvert^{2} - \Rc(X, X)\Big) \dvol_{g} &= - \frac{2}{m} \int_{M} \lvert X \rvert^{2} \Div(X) \dvol_{g} + (n - 2) \int_{M} \langle X , \nabla R \rangle \dvol_{g}.\\ &\hspace*{1cm} - (n - 2) \int_{M} \langle X , \nabla \rho \rangle \dvol_{g}.
\end{align*}
On the other hand, from Yano's integral formula \eqref{eq:Yano-integral-formula} we have
\begin{align*}
\int_{M} \Big(\lvert\nabla X \rvert^{2} - \Rc(X, X)\Big) \dvol_{g} = \frac{1}{2} \int_{M} \lvert \sL_{X} g \rvert^{2} \dvol_{g} - \int_{M} \Div(X)^{2} \dvol_{g}.
\end{align*}
Combining above expressions, it yields that
\begin{align*}
\frac{1}{2} \int_{M} \lvert \sL_{X} g \rvert^{2} \dvol_{g} &= \int_{M} \Div(X)^{2} \dvol_{g} - \frac{2}{m} \int_{M} \lvert X \rvert^{2} \Div(X) \dvol_{g} \\ &\hspace*{1cm} + (n - 2) \int_{M} \langle X , \nabla R \rangle \dvol_{g} - (n - 2) \int_{M} \langle X , \nabla \rho \rangle \dvol_{g}.
\end{align*}
Next, integrating \eqref{eq:identity-after-taking-trace-and-applying-nabla-and-inner-prod-with-X} and using the divergence theorem and formulae \eqref{eq:two-divergence-formulae} we get
\begin{align*}
\frac{1}{m} \int_{M} \lvert X \rvert^{2} \Div(X) \dvol_{g} &= (n - 1) \int_{M} \langle X , \nabla R \rangle \dvol_{g} + \int_{M} \Div(X)^{2} \dvol_{g} \\ &\hspace*{0.5cm} - n \int_{M} \langle X , \nabla \rho \rangle \dvol_{g}.
\end{align*}
It implies that
\begin{align*}
\frac{1}{2} \int_{M} \lvert \sL_{X} g \rvert^{2} \dvol_{g} = - \int_{M} \Div(X)^{2} \dvol_{g} - n \int_{M} \langle X , \nabla R \rangle \dvol_{g} + (n + 2) \int_{M} \langle X , \nabla \rho \rangle \dvol_{g}.
\end{align*}
Now, for each of the criterion we see that the right-hand side is non-positive. Note that under the third criterion $(iii)$ the middle term on the right-hand side is in fact zero (see \eqref{eq:integral-condition-obtained-by-BourEzin-theorem}). It implies that we must have $\sL_{X} g = 0$, i.e. hence $X$ is a Killing vector field. So, we get the desired result.
\end{proof}
	
We are now in position to prove Proposition \ref{prop:div-free-plus-upper-bound-on-rho-implies-triviality}.
	
\begin{proof}[{\textbf{Proof of Proposition \ref{prop:div-free-plus-upper-bound-on-rho-implies-triviality}}}]
Since $\Div(X) = 0$, the identity \eqref{eq:identity-after-taking-trace-of-gen-mQE-Yam-type} becomes
\begin{align}\label{eq:identity-taking-trace-gen-mQE-Yam-type-with-div-zero}
n \rho = (n - 1) R + \frac{1}{m} \lvert X \rvert^{2}.
\end{align}
Next, integrating the identity \eqref{eq:evaluating-gen-m-QE-Yam-type-on-the-pair-X-X} on the manifold $M$ (which is closed) and applying the divergence theorem and using the second formula in \eqref{eq:two-divergence-formulae} one obtain
\begin{align*}
\int_{M} \Rc(X, X) \dvol_{g} = \frac{1}{m}\int_{M} \lvert X \rvert^{4} \dvol_{g} + \int_{M} (R - \rho)\lvert X \rvert^{2} \dvol_{g}.
\end{align*}
On the other hand, since $\Div(X) = 0$, integrating \eqref{eq:Bochner-type-identity-for-gen-m-QE-Yam-type} we have
\begin{align*}
\int_{M} \lvert\nabla X \rvert^{2} \dvol_{g} = \int_{M} \Rc(X, X) \dvol_{g}.
\end{align*}
It implies that
\begin{align*}
\int_{M} \lvert\nabla X \rvert^{2} \dvol_{g} = \int_{M} \left(\frac{1}{m} \lvert X \rvert^{2}  + R - \rho\right)\lvert X \rvert^{2} \dvol_{g}.
\end{align*}
Now using \eqref{eq:identity-taking-trace-gen-mQE-Yam-type-with-div-zero} and simplifying we obtain
\begin{align*}
\int_{M} \lvert\nabla X \rvert^{2} \dvol_{g} = \frac{1}{n-1} \int_{M} \left(\frac{n-2}{m} \lvert X \rvert^{2} + \rho\right)\lvert X \rvert^{2} \dvol_{g}.
\end{align*}
Then the condition \eqref{eq:upper-bound-assumption-for-rho} implies that $X$ must vanishes identically on $M$.
\end{proof}
	
We remark that the proof of Corollary \ref{cor:div-free-plus-scal-curv-nonpositive-implies-triviality} immediately follows from the identity \eqref{eq:identity-taking-trace-gen-mQE-Yam-type-with-div-zero} and Proposition \ref{prop:div-free-plus-upper-bound-on-rho-implies-triviality}. Also, note that the Corollary \ref{cor:div-free-plus-scal-curv-nonpositive-implies-triviality} under the second assumption (namely, $m < 0$ and $R \leq \rho$) is basically the Corollary E in \cite{CarvLimCosta-Filho2026} for generalized $m$-QE manifolds of Yamabe-type.

\subsection{Proof of the results in the non-compact case}
\label{subsec:proof of main results_non-compact case}
	
In this sub-section we prove our result when the manifold is complete and non-compact.
	
\begin{proof}[{\textbf{Proof of Theorem \ref{thm:triviality-of-gen-m-QE-mfd-complete-noncompact-case}}}]
Since the quadruple $(M^n, g, X, \rho)$ is a generalized $m$-QE manifold of Yamabe-type, from the relation \eqref{eq:identity-after-taking-trace-of-gen-mQE-Yam-type} we have
\begin{align*}
\Div(X) = (n - 1) R - n \rho + \frac{1}{m} \lvert X \rvert^{2}.
\end{align*}
Then from the hypothesis \eqref{eq:condition-implies-X-is-zero-complete-noncompact-case} we see that either $\Div(X) \geq 0$ or $\Div(X) \leq 0$, i.e. $\Div(X)$ does not change sign on $M$. Now, since $\lvert X \rvert \in L^1(M)$, by applying Proposition \ref{prop:Caminha-complete-noncompact-case} we get that $\Div(X) = 0$. In particular, we have
\begin{align*}
(n - 1) R - n \rho + \frac{1}{m} \lvert X \rvert^{2} = 0.
\end{align*}
Once again using \eqref{eq:condition-implies-X-is-zero-complete-noncompact-case} we conclude that $X$ vanishes identically on $M$. Therefore, $(M^n, g)$ must be an Einstein manifold since the dimension $n \geq 3$.
\end{proof}
	
We finally prove our Theorem \ref{thm:parallel-vec-field-conclusion-gen-mQE-mfd-Yam-type-non-compact-case}.

\begin{proof}[{\textbf{Proof of Theorem \ref{thm:parallel-vec-field-conclusion-gen-mQE-mfd-Yam-type-non-compact-case}}}]
Firstly, we re-write the identity \eqref{eq:identity-after-taking-trace-and-applying-nabla-and-inner-prod-with-X} as follows
\begin{align*}
\langle X, \nabla\rho \rangle = (n - 1) \langle X, \nabla(R - \rho) \rangle - \langle \nabla\Div(X), X \rangle + \frac{1}{m} X (\vert X \vert^2).
\end{align*}
Now, multiplying both sides by $(n - 2)$ and using the identity \eqref{eq:Bochner-type-identity-for-gen-m-QE-Yam-type} we get
\begin{align*}
(n - 2) \langle X, \nabla\rho \rangle &= (n - 1)\vert\nabla X \vert^{2} - (n - 1)\Rc(X, X) + \frac{2(n - 1)}{m}\vert X \vert^{2}\Div(X) - \frac{n-1}{2}\Delta\vert X \vert^{2} \\
&\hspace*{1cm} - (n - 2)\langle \nabla\Div(X), X \rangle + \frac{n - 2}{m} X (\vert X \vert^2).
\end{align*}
Next, we apply the divergence formulae \eqref{eq:two-divergence-formulae} and simplifying we obtain
\begin{align*}
(n - 2) \langle X, \nabla\rho \rangle &= (n - 1)\vert\nabla X \vert^{2} - (n - 1)\Rc(X, X) + (n - 2)\Div(X)^{2} - \frac{n}{m}X(\vert X \vert^{2}) \\
&\hspace*{0.8cm} + \Div \left( \frac{2(n - 1)\vert X \vert^{2}}{m}X - (n - 2)\Div(X)X \right) - \frac{n-1}{2}\Delta\vert X \vert^{2}.
\end{align*}
Since $\Delta\vert X \vert^{2} = 2 \Div\big(\vert X \vert \cdot \nabla \vert X \vert\big)$, this is further simplified to the following formula
\begin{align*}
&\Div\left((n - 2)\Div(X)X + (n - 1)\vert X \vert \cdot \nabla \vert X \vert - \frac{2(n - 1)\vert X \vert^{2}}{m}X\right)\\ &= (n - 1)\vert\nabla X \vert^{2} - (n - 1)\Rc(X, X) + (n - 2)\Div(X)^{2} - \frac{n}{m}X(\vert X \vert^{2}) - (n - 2) \langle X, \nabla\rho \rangle.
\end{align*}
In other words, if we define a smooth vector field $V \in \mathfrak{X}(M)$ by
$$ V := (n - 2)\Div(X)X + (n - 1)\vert X \vert \cdot \nabla \vert X \vert - \frac{2(n - 1)\vert X \vert^{2}}{m}X, $$
then we have
\begin{equation}\label{eq:div-formula-for-new-vector-field-V}
\begin{split}
\Div(V) &= (n - 1)\vert\nabla X \vert^{2} + (n - 2)\Div(X)^{2} \\ &\hspace*{1cm} - (n - 2) \langle X, \nabla\rho \rangle - (n - 1)\Rc(X, X) - \frac{n}{m}X(\vert X \vert^{2}).
\end{split}
\end{equation}
But then the condition \eqref{eq:condition-to-guarantee-div-of-V-is-nonnegative} implies that $\Div(V) \geq 0$. On the other hand, using the Minkowski inequality we get
\begin{align*}
\frac{1}{n - 1}\Vert V \Vert_{L^{1}(M)} &\leq \frac{n - 2}{n - 1}\Vert \Div(X) X \Vert_{L^{1}(M)} + \Vert \vert X \vert \cdot \nabla \vert X \vert \Vert_{L^{1}(M)} + \frac{2}{\vert m \vert} \Vert \vert X \vert^{2} \cdot X \Vert_{L^{1}(M)} \\
&= \frac{n - 2}{n - 1}\Vert \Div(X) X \Vert_{L^{1}(M)} + \Vert \vert X \vert \cdot \nabla \vert X \vert \Vert_{L^{1}(M)} + \frac{2}{\vert m \vert} \Vert X \Vert_{L^{3}(M)}^{3}.
\end{align*}
Now, applying the H\"older inequality and using one of the hypothesis conditions satisfied by $X$ we see that $\vert V \vert \in L^{1}(M)$. Therefore, applying Proposition \ref{prop:Caminha-complete-noncompact-case} we must have $\Div(V) = 0$ on $M$. Then from \eqref{eq:div-formula-for-new-vector-field-V} we obtain
\begin{align*}
(n - 1)\vert\nabla X \vert^{2} + (n - 2)\Div(X)^{2} = (n - 2) \langle X, \nabla\rho \rangle + (n - 1)\Rc(X, X) + \frac{n}{m}X(\vert X \vert^{2}).
\end{align*}
Once again using the condition \eqref{eq:condition-to-guarantee-div-of-V-is-nonnegative} we conclude that $\nabla X = 0$ and hence $X$ is a parallel vector field. This completes the proof.
\end{proof}

\vspace*{2mm}
\section*{Acknowledgements}
	
This work was supported in part by an Institute Post Doctoral Fellowship (IPDF) from the Indian Institute of Technology Bombay.

\vspace*{2mm}
\section*{Conflict of Interest Statement}

The author declare no conflict of interest.

\vspace*{2mm}
\section*{Data Availability Statement}
	
Data sharing not applicable to this article since no datasets were generated or analyzed during the current study.

\vspace*{2mm}

\end{document}